\documentclass[11pt]{article}
\usepackage[utf8]{inputenc}
\usepackage{amsfonts, amssymb, amsmath, amsthm}
\usepackage{color}

\usepackage{srcltx}

\usepackage{epsfig,graphicx}
\usepackage[margin=1.25in]{geometry}

\usepackage{enumerate}

\newtheorem{Thm}{Theorem}[section]
\newtheorem{Lem}[Thm]{Lemma}

\newtheorem{Cor}[Thm]{Corollary}
\theoremstyle{remark}
\newtheorem{Def}[Thm] {Definition}
\newtheorem{Rem}[Thm] {Remark}

\newtheorem{Exa}[Thm] {Example}

\newcommand{\length}{\operatorname{\length}}

\def\length{\operatorname{length}}

\def\RR{\mathbb{R}}

\def\RR{{\mathbb R}}

\newtheoremstyle{mydefinition}{}{}{\normalfont}{0pt}{\scshape}{.}{.5em}{}
\theoremstyle{mydefinition}

\newtheoremstyle{myremark}{}{}{\small\normalfont}{0pt}{\small\scshape}{.}{.5em}{}
\theoremstyle{myremark}

\numberwithin{equation}{section}

\def\beq{\begin{equation}}
\def\eeq{\end{equation}}
\def\beqn{\begin{equation*}}
\def\eeqn{\end{equation*}}
\def\beqy{\begin{eqnarray}}
\def\eeqy{\end{eqnarray}}
\def\beqyn{\begin{eqnarray*}}
\def\eeqyn{\end{eqnarray*}}

\def\length{{\rm length}}

\def\b1{{\boldsymbol{1}}}

\def\IR{{\mathbb R}}

\def\eps{\varepsilon}

\def\eps{\varepsilon}

\title{Stable and non-symmetric pitchfork bifurcations}

\author{Enrique Pujals , Michael Shub \footnote{ This work  is partially supported 
by the Smale institute.}  and Yun Yang}

\begin{document}

\maketitle

\begin{abstract}In this paper, we present a criterion for pitchfork bifurcations of smooth vector fields based on a topological argument.     Our result expands Rajapakse and Smale's result \cite{RS2} significantly.  Based on our criterion, we present  a class of families of stable and non-symmetric vector fields undergoing a
pitchfork bifurcation.  
\end{abstract}

%

\section{Introduction}
In this paper we consider the bifurcation of the isolated equilibria of the locally defined vector fields in $\RR^n$. This well studied subject has recently had some fresh observations by Rajapakse and Smale \cite{RS2} concerning the pitchfork bifurcation and its relevance for biology. It is our intention to expand their treatment by showing that there are significantly new subtleties.

A phenomenon is called observable if it is stable under small perturbation.   The dogma of bifurcation theory reasonably asserts that the dynamics of the vector fields and their bifurcations used to explain the observable phenomenon should be stable as well. 
It is known that the only generic and stable simple non-hyperbolic bifurcation with one-dimensional parameter is the saddle-node bifurcation, in which zeros of adjacent indices are created or cancelled. Hence the pitchfork bifurcation known as the transition from a single stable equilibrium to two new stable equilibria separated by a saddle is not generally stable.

While the pitchfork bifurcation is not generally stable, it is stable under certain additional hypothesis such as symmetry (namely equivariant branching) or the vanishing of a certain second derivative at the bifurcation point (\cite{R1};Theorem 7.7, \cite{K},etc.).  The stability and the symmetry of the pitchfork bifurcation is usually expressed in terms of its normal form 
  $\dot{u}=u\eps-u^3.$
This family of vector field is invariant under the involution $u\rightarrow -u$.   
Rajapakse and Smale \cite{RS1,RS3,RS2} are most interested in the case when one stable equilibrium gives rise to two new stable equilibria after the bifurcation and without symmetry.   They argue that if the state of a cell is modeled as a stable equilibrium, then the cellular division should give rise to two new stable equilibria after division. They model this phenomenon with a non-symmetric pitchfork bifurcation in which one stable equilibrium gives rise to three, two new stable and one unstable.  We generalize their results significantly and supply complete proofs.  

Consider an one-parameter family of $C^{2}$ vector fields in $\IR^n$ given by
$\dot{x} = V(x, \eps)$ where $\eps\in \IR^1.$
A point $(x_0,\eps_0)$ is \emph{simple non-hyperbolic}  if $D_x V(x_0, \eps_0)$ has a simple eigenvalue $\lambda=0$
and all other eigenvalues are not on the imaginary axis. A fixed point $(x,\varepsilon)=(x_0,\eps_0)$  is said to undergo a \emph{$1\rightarrow \text{many}$ bifurcation}, if the flow has one and only one fixed point in a neighborhood of $x_0$  for any sufficient close $\varepsilon\leq\eps_0$  while the flow has many fixed points around $x_0$ for any sufficient close $\eps>\eps_0$. A fixed point $(x,\varepsilon)=(x_0,\eps_0)$  is said to undergo a \emph{$\text{many }\rightarrow 1$ bifurcation}, if the flow has many fixed points in a neighborhood of $x_0$  for any sufficient close $\varepsilon\leq\eps_0$ while the flow has one and only one fixed point around $x_0$ for any sufficient close $\eps>\eps_0$. We say that the bifurcation is of {\em pitchfork-type}  if there is a neighborhood of $x_0$ such that $x_0$ is the unique non-hyperbolic zero in the neighborhood for any sufficient cloase $\eps\leq \eps_0$, $x_0$ continues smoothly $x_{\eps}$ as one of the equilibrium points for any sufficient close $\eps>\eps_0$ and the number of the equilibrium points for any sufficient close $\eps>\eps_0$ is greater than or equals to three. Moreover it is called a {\em pitchfork bifurcation} if the number of new equilibrium for pitchfork-type bifurcation is exactly three. 

In the literature, the hypotheses to guarantee the existence of a pitchfork bifurcation generally contain one of the following two types assumptions:

\noindent{\em Type a}. The set of zeros $(x_{t},,\eps_t)$  consists of one stable equilibrium of index $-1$  for $t<0$ which continue smoothly to $(x_t,\eps_t)$ of index $1$ for $t>0$, i.e. an eigenvalue at the zeros changes from negative to positive. See \cite{CR1,CR2,CH,RS2,LSW}.

\noindent{\em Type b}. The equation has some symmetry which is frequently exhibited by a normal form with respect to a center manifold which is assumed to be explicitly known.  See Chapter 7 in \cite{K}, Chapter 19 in \cite{W} and references therein. 

In this paper, we assume neither of these scenarios. We work with $n$-dimensional vector fields and prove that Type a follows from our hypotheses. Our hypotheses are much easier to check compared with the hypotheses of Type a or Type b. We also give examples without symmetry and examples to show that all of the hypotheses are necessary for the existence of pitchfork-type and  pitchfork bifurcations. 

 An essential part of our treatment relies on a topological argument.  We refer to \cite{KS,N, R, S} for some work in the literature using topological approaches in dealing with bifurcation problems.  Here we consider under which conditions the bifurcation of an isolated simple non-hyperbolic equilibrium with non-zero index  gives rise to many equilibria  with non-zero index. We are interested in the bifurcation of stable equilibria which are interior to the basin of attraction.
Our criteria for the bifurcation are multidimensional (See (P0)-(P2) below) and are expressed in terms of the derivative at the bifurcation point. We do not invoke the explicit form of a reduction to the center manifold, even for (P3).  Based on our  criterion, we give an example of a family of vector fields without symmetry which undergoes a pitchfork bifurcation.

Fix $(x_0,\eps_0)$.  Denote by $\mathcal{F}$ to be the set of one parameter vector fields
  $V\in\mathcal{F}$ such that it satisfies the following conditions:
  \begin{enumerate}
\item[(P0)] $V(x,\eps_0)$  has an isolated simple non-hyperbolic equilibrium $x_0$ with non-zero index. 
\item[(P1)]  $\frac{\partial V}{\partial\eps}\vert_{(x_0,\eps_0)}\in \text{image}(D_xV);$
\item[(P2)] there exists $\omega=(\omega_1,\cdots,\omega_{n+1})^{\top}$ such that 
$$DV(x_0,\eps_0)\omega=0\text{ and }\omega_{n+1}\neq 0,
\text{ and } D(\text{det}(D_xV))(x_0,\eps_0)\omega\neq 0.$$
\end{enumerate}
Our conditions are easy to check. Here we make some comments.
\begin{itemize}
\item (P0) is immediate if there exists a small ball $B(x_0)$ around $x_0$ such that for any $\eps<\eps_0$ close enough to $\eps_0$,  there is one and only one zero  inside the small ball $B(x_0)$ and it is transversal, i.e., $0$ is not an eigenvalue of the determinant.  Then the index of the zero at $\eps=\eps_0$ is either $1$ or $-1$ (See Section 2.1). 
\item (P1) is verified if the rank of the derivative $DV(x_0,\eps_0)$ is $n-1$. 
\item In coordinates $(x_1,\cdots, x_n)$ given by the eigenspaces of $D_xV(x_0,\eps_0)$ with $(1,0,\cdots,0)$ the zero eigenvector, (P2) is true iff $\frac{\partial^2 V_1}{\partial x\partial \eps}\neq 0.$
\end{itemize}
\begin{Thm}[Bifurcation]\label{BB}Every $V\in \mathcal{F}$  undergoes a pitchfork-type bifurcation, i.e.,  it is a  $1\rightarrow k$ or $k\rightarrow 1, 3\leq k\leq+\infty $ bifurcation at $(x_0,\eps_0)$. 
\end{Thm}
Theorem \ref{BB} implies Rajapakse and Smale's result. 
\begin{Cor}[\cite{RS2}] \label{RScor}Suppose the following conditions:
\begin{enumerate}
\item $\frac{dx}{dt}=V(x,\eps), x\in X, V(x_0,\eps)=0$ for $\varepsilon\leq \varepsilon_0$ and the determinant of the Jacobian of $V$ at $(x_0,\varepsilon_0) $ is zero. 
\item the eigenvalue of the Jabobian matrix satisfy  
\begin{eqnarray*}
&&\text{real}(\lambda_i)<0, i>1;
\lambda_1=0 \text{ and } \frac{d\lambda_1}{d\varepsilon}\vert_{(x_0,\varepsilon_0)}>0.
\end{eqnarray*}
\item the multiplicity of $V(x,\eps_0)$ at $x_0$ is three and the Poincar\`{e}-Index is $(-1)^n$ relative to a disk $B^n_r$ about $x_0$.
\end{enumerate}
These are sufficient conditions for the pitchfork bifurcation.
\end{Cor}
The condition (P1) is trivial in Corollary \ref{RScor} since $x_0$ is the only zero points for any $\varepsilon\leq \varepsilon_0$. Conditions (P0) and (P2) are also trivial in Corollary \ref{RScor}. The multiplicity assumption in Corollary \ref{RScor} implies the bifurcation given in Theorem \ref{BB} is exactly one to three.  Hence Theorem \ref{BB} is much more general. We refer to Section 2 for an example where P0,P1,P2 are not trivial while Theorem \ref{BB} applies. Moreover, corollary \ref{RScor} may be false if the conditions are not satisfied (see section \ref{failing}).

As one may have noticed that one of the key points in Theorem \ref{BB} is that we consider the derivative of the determinant of $DV$, instead of $V, DV, D^2V$ as the classical argument goes.  The proof of theorem \ref{BB} goes in two steps. Here is an outline:

\noindent{\em Step 1}. From the fact that the equilibrium  is a simple non-hyperbolic point, it follows the there is a center manifold normally hyperbolic associated to it. Moreover, it is  shown that the index property can be reduced to the index restricted to the center manifold. 

\noindent{\em Step 2}. P0 and P1 guarantee a continuation of the zero to any parameter value near the bifurcation parameter. This follows from considering the dynamics of vector fields along the center manifolds and then proving  the fact that P0 and P1 are carried over to the dynamics along the center manifold. Moreover, P2, i.e., the condition on the derivative of the determinant implies that the bifurcation is of pitchfork-type, meaning that at least two new zeros with different indices arise after the bifurcation. 

A natural question is to consider how many equilibria appear in Theorem \ref{BB}.   
  The following theorem gives a  criterion for the existence of a pitchfork bifurcation which doesn't depend on the multiplicity hypothesis in Corollary \ref{RScor}.  
Let $\mathcal{G}\subset \mathcal{F}$ be such that any $V\in \mathcal{G}$ is of the form:
$\begin{cases}\dot{u}=F(u,y,\eps)\\
\dot{y}=My+G(u,y,\eps)
\end{cases}$
where $u\in\mathbb{R}^1$ and $y\in\mathbb{R}^{n-1}$, the square matrix $M$ has eigenvalues with only non-zero real parts and 
$$F(0,0,0)=0, DF(0,0,0)=0, G(0,0,0)=0, DG(0,0,0)=0,$$
satisfy one extra condition:
 \begin{enumerate}
\item[(P3)]
$D_{uu}\det DV\vert_{(0,0,0)}-D_y(\det DV)(M^{-1}G_{uu})\vert_{(0,0,0)}\neq 0.$
\end{enumerate}
We give a comment here about (P3). 
\begin{itemize}
\item In the case that the center manifold is explicitly known and  in the presence of conditions P0,P1, P2, (P3) is equivalent to the usual hypotheses that the third derivative of $V$ restricted to the center manifold is not zero. But we re-emphasis that (P3) doesn't require knowledge of the center manifold which might be difficult to compute.   
\end{itemize}

\begin{Thm}[Pitchfork Bifurcation]\label{EP}Every $V\in \mathcal{G}$ undergoes a $1\rightarrow 3$ or $3\rightarrow 1$ bifurcation.
\end{Thm}
Endow $\mathcal{F}$ with the usual topology of $C^{\infty}$ maps. Based on Theorem \ref{EP}, we obtain the genericity of the pitchfork bifurcation.
\begin{Thm}\label{PG}The set of vector fields with pitchfork bifurcation is open and dense in $\mathcal{F}$. 
\end{Thm}
It is worth noting that there is a Banach space version of these theorems where index refer to index in the finite dimensional center manifold.

This paper is organized as follows. In Section 2, we give some examples: one of them  shows the lack of stability  of the pitchfork bifurcation under the general perturbation and one shows the existence of pitchfork bifurcation without symmetry.  In Section 3, we provide examples that show if any of the assumptions fail, there may be not pitchfork bifurcation and therefore show the necessity of our assumptions. 
 In Section 4, we introduce some preliminaries. As a preparation for the proof of Theorem \ref{BB}, we give some discussion on the index of fixed points in Section 5. Then in Section 6 we deliver some observations for the one-dimensional case.  In Section 7, we present Theorem \ref{BB} based on  center reduction techniques and the product property of the index of fixed points. We  also give the proof of Corollary \ref{RScor}.  In Section 8, we give the proof of Theorem \ref{EP} and Theorem \ref{PG} based on an analysis of graph transform.

To finish the introduction, we note that the conditions we are proposing are more general (less restrictive) to the ones available in the literature. We give sufficient and necessary conditions for the existence of stable pitchfork bifurcations in terms of the Taylor expansion of $V$ only at the point $(x_0,\eps_0)$. This makes the conditions significantly easier to check.   
In the previous studies of the pitchfork bifurcation, for instance the one provided by Crandall and Rabinowitz (see \cite{CR1, CR2}) and explained in  section 6.6 of the book “Methods of bifurcation” by Shui-Nee Chow and Jack Hale (see \cite{CH}) it is explicitly assumed that  for any parameter nearby the bifurcation one there is at least one zero, that is, there is a branch of solutions through $(0,0)$. In particular, that hypothesis is not assumed in our paper. Moreover, it is shown in example 3.3 that even assuming there is a branch of solutions, if the other condition (P2)  is not  satisfied then it could happen that there is no bifurcations. Also,  example 3.2 shows that the conditions provided in  \cite{RS1} is not  enough to guarantee a bifurcation if the zero of the initial vector field is allowed to move.

\section{Examples}
In this section, we will use our method to detect bifurcations. Compared with the classical method--normal form, our method tends to be more efficient. We also give a construction of a one-parameter family of vector fields without symmetry which undergoes a pitchfork bifurcation.

\begin{Exa}[Revisiting the Rajapakse-Smale example]\label{RS}Consider $$
\begin{cases}
& \dot{x}  =  y^2 -(\eps+1) y -x, \\
& \dot{y}  =  x^2 -(\eps+1) x- y.
\end{cases}
$$
near the equilibrium point $(x,y)=(0, 0)$ at the bifurcation parameter $\eps=0$.
We use Theorem \ref{BB} to verify the existence of bifurcation:  for $\eps<0$, around $(0,0)$
there is one and only one real equilibriums: $(x_0, y_0)=(0, 0)$. Moreover,
$
\det(DV)=1-(2x-(\eps+1))(2y-(\eps+1)).
$
Then 
$
\dfrac{\partial}{\partial \eps} \det(D_x V)|_{((x,y), \eps)=((0, 0), 0)} 
= -2. $
Now let's verify the condition (P2):
$\frac{\partial V}{\partial \eps}(0,0)=(0,0).$
Hence we have all of the conditions in Theorem \ref{BB} for Example \ref{RS}. Since the multiplicity of $(0,0)$ is three (one zero far away from $(0,0)$), we have a  pitchfork bifurcation. 

Here we also give the argument using the classical method, Normal form, as a comparison.  Under the change of coordinates 
$\begin{cases} u=x-y \\ v=x+y \end{cases}$, 
we get
$
\begin{cases}
\ \ \dot{u}  =  \eps u - uv, \\
\ \ \dot{v}  =  -(2+\eps) v + \frac{u^2+v^2}{2}. 
\end{cases}
$
For $\eps$ near 0, we reduce this vector field to a parametrized equation 
along the local center manifold, that is, 
$
 \dot{u} =  \eps u - u h(u, \eps), 
$
where $v=h(u, \eps)$ satisfies that $h(0, 0)=0$, $D_{(u, \eps)} h(0, 0)=0$, 
and 
$
\partial_u h(u, \eps) [\eps u - u h(u, \eps)] = -(2+\eps) h(u, \eps) + \frac{u^2+(h(u, \eps))^2}{2}. $
Taking $\eps=0$, and expanding $h(u, 0)=h_2 u^2 + O(u^3)$, we get $h_2=\frac14$. Therefore, we obtain
$
\dot{u} =  \eps u - \frac14 u^3.$
By Lemma \ref{onethree}, this vector field
 experiences a pitchfork bifurcation. 
\end{Exa}
Even though Example \ref{RS} doesn't have $(x,y)\rightarrow (-x,-y)$ symmetry, it does have center symmetry, i.e. $(x,y)\rightarrow (y,x).$
Here we would like to add a small perturbation of the Rajapakse-Smale example to destroy the symmetry.  We recall the definition of symmetry for a vector field. 
\begin{Def}[P278, \cite{K}]We say the  vector field
$\dot{x}=V(x,\eps),x\in\mathbb{R}^n,\eps\in\mathbb{R},$
has {\it symmetry} if there exists a matrix transformation $R:x\mapsto Rx$ satisfies:
$$RV(x,\eps)=V(Rx,\eps), R^2=I.$$
\end{Def}

\begin{Exa}[Pitchfork bifurcation without symmetry.]\label{RS2}Consider the 2-D ODE 
$$
\begin{cases}
& \dot{x}  =  y^2 -(\eps+1) y -x, \\
& \dot{y}  =  x^2 -(\eps+1) x- y+\eps^2.
\end{cases}
 $$
near the equilibrium point $(x,y)=(0, 0)$ at the bifurcation parameter $\eps=0$.
%
We note here that Corollary \ref{RScor} doesn't apply to this example. We use Theorem \ref{BB}  to verify the pitchfork bifurcation:  for $a<1$, 
there are only one equilibrium: $(x_0, y_0)=(0, 0)$ in a neighborhood of $(0,0)$. Moreover, $
\det(DV)=1-(2x-(\eps+1))(2y-(\eps+1)). $
Hence $\det(DV)=0$ at $((0,0),0)$ and
$\dfrac{\partial}{\partial \eps} \det(D_x V)|_{((x,y), \eps)}=-2. $
Hence we have all of the conditions in Theorem \ref{BB} for Example \ref{RS2}. Since the multiplicity of $(0,0)$ is three (one zero far away from $(0,0)$),  the pitchfork bifurcation follows. 
\end{Exa}
\begin{Exa}[Perturbation of the pitchfork bifurcation]
Consider the following family of vector fields:
\[\begin{cases}&\dot{x}=y^2-(\eps+1)y-x;\\
&\dot{y}=(1+\eps_0)x^2-(\eps+1)x-y.\end{cases}\]
Let $y^2-(\eps+1)y-x=0$ and $(1+\eps_0)x^2-(\eps+1)x-y=0$. Then we have $x=y^2-(\eps+1)y.$  Plugging it into the second one  at $\eps=0$ gives
$(1+\eps_0)(y^2-y)^2-(y^2-y)-y=0,$
i.e. $y^2((1+\eps_0)y^2-2(1+\eps_0)y+\eps_0)=0.$
Hence as long as $\eps_0\neq0$, we have four zeros 
$y=0,y=0, y=\frac{1+\eps_0\pm\sqrt{1+\eps_0}}{1+\eps_0}.$ Hence the vector field can only undergo a saddle-node bifurcation at $(0,0)$ while for $\eps_0=0$, we already know it undergoes a pitchfork bifurcation. 
We can view this as the perturbation of the Rajapakse and Smale example. As long as $\eps_0\neq 0$, the vector fields undergoes a saddle-node bifurcation which maybe hard to see numerically.  When $\eps_0=0$, it undergoes a pitchfork bifurcation. This shows clearly that pitchfork bifurcation is not stable. Moreover, it can't be because the derivative at the bifurcation point in $(x,\eps)$ has two dimensional kernel so the bifurcation can not be transversal to the zero section which is also clearly visible from the fact that the zero set is not locally a manifold. 
\end{Exa}

\section{Necessity of the conditions provided}
\label{failing}

In the present section, we show through examples that if any of our conditions are not satisfied then there is not pitchfork bifurcation.
\begin{Exa}[Missing (P0): $3\rightarrow 0\rightarrow3$ bifurcation]The vector fields 
$$\dot{x}=\eps x+x^2+x^3$$
has $3\rightarrow 0\rightarrow3$  bifurcation. 
 The zeros are given by $x=0$ and $x=\frac{-1\pm\sqrt{1-4\eps}}{2}$. We miss (P0) because the index of $x=0$ is zero. Even though the other conditions (P1),(P2) and (P3) are all satisfied, we don't have pitchfork bifurcation in this example.
 \end{Exa}
\begin{Exa}[Missing (P1): No bifurcation]The vector fields
$$\dot{x}=\eps-\eps x+x^3$$
has no bifurcation. 
It is easy to see that (P0), (P2) and (P3) all hold, but (P1) does not. There is only one solutions for small $\eps$.
This is because if (P1) does not hold, then the zeros lie on a smooth curve through $(x_0,\eps_0)$. So there is no bifurcation. In this example, the eigenvalue of the zeros go from positive to zero to positive. 
\end{Exa}
%
%

 \begin{Exa}[Missing (P2):  A moving center manifold. ]Consider
$$\begin{cases}
 &\dot{x}=x\eps+2xy+x^3\\
 &\dot{y}=2y+\eps
 \end{cases}.$$
We have 
$\frac{\partial V}{\partial \eps}=(0,1)$
which is transversal to the center direction $(1,0)$. However, this is not enough. 
Also $\frac{\partial \text{det}(DV)}{\partial \eps}(0,0)=\frac{\partial (\eps+2y+3x^2)}{\partial \eps}=2\neq 0.$
However, there is no bifurcation. The only equilibrium is 
$(0,-\frac{\eps}{2}).$ This is because the (P2) condition is not satisfied: the kernel of $DV$ is generated by $(1,0,0)$ and $(0,-\frac{1}{2},1)$. So 
$D(\det DV)(0,1,0)=0, D(\det DV)(0,-\frac{1}{2},1)=0.$
\end{Exa}
 

\begin{Exa}[Missing (P3): $1\rightarrow k, k>3$ bifurcation]Consider the vector fields:
$$\dot{x}=x(\eps-x^2\sin^2\frac{1}{x}-x^4).$$
We claim that this example satisfies  the conditions in our Main Theorem.  Now let's prove this claim. When $\eps=0$, 
$V(x,0)=-x^3\sin^2\frac{1}{x}-x^5.$
Since $x^2\sin^2\frac{1}{x}+x^4>0, \forall x\neq 0,$ 
we have 
$
V(x,0)=-x(x^2\sin^2\frac{1}{x}+x^4)
<0, \forall x>0.$
Similarly, we have $
V(x,0)>0, \forall x<0.$
Hence the index of $(0,0)$ is $-1.$
We have:
$\frac{\partial V(x,\eps)}{\partial \eps}(0,0)=0,$ and 
$\frac{\partial V(x,\eps)}{\partial x}=(\eps-x^2\sin^2\frac{1}{x}-x^4)+x(4x^3-2\sin\frac{1}{x}\cos\frac{1}{x}+2x\sin^2\frac{1}{x}).$
Hence $\frac{\partial V(x,\eps)}{\partial x}(0,0)=0,$
and $\frac{\partial^2 V(x,\eps)}{\partial \eps\partial x}=1\neq0.$  So it satisfies (P0),(P1) (P2) but not (P3). It undergoes a $1\rightarrow k,k>3$ bifurcation. One direct way to prove it is the compute the zeros for the vector field numerically. 
\end{Exa}
\begin{Exa}[Missing ``Half of  (P3)": $1\rightarrow k, k>3$ bifurcation]Consider
$$\begin{cases}&\dot{x}=2x^3-xy+xy^2-4x^5+x(\eps-x^4\sin^2\frac{1}{x}-x^6)\\
&\dot{y}=2y-4x^2\end{cases}.$$
This vector field has the same zeros as 
$\dot{x}=x(\eps-x^4\sin^2\frac{1}{x}-x^6)$
which undergoes a 1 to $k,k>3$ bifurcation. 
Even though it satisfies
$D_{uu}(\det DV)(0)=8$ 
is positive definite, it doesn't satisfy (P3).  This is because 
$(0, -M^{-1}G_{uu}(0))=(0,4)$ and $D(\det(D_{(x,y)}V(0)))=(0,-2)$.
Hence $D_{uu}(\det DV)(0)+D_y(\det(D_{(x,y)}V)))(0, -M^{-1}G_{uu}(0))^{\top}=0.$
\end{Exa}

\section{Preliminaries}

    \subsection{An index property for vector fields}
 Given a map $\phi: S^n\rightarrow S^n$, the degree of $\phi$ denoted by $\text{deg}\phi$ is the unique integer such that for any $x\in H_nS^n$, 
  $\phi_{\ast}(x)=\deg\phi \cdot x.$ 
  Here $\phi_{\ast}$ is the induced homomorphism in integral homology.  
  Suppose that $x_0$ is an isolated zero of the vector field $V$. Pick a closed disk $D$ centered at $x_0$, so that $x_0$ is the only zero of $V$ in $D$. Then we define the index of $x_0$ for $V$, $\text{ind}_{x_0}(V)$, to be the degree of the map
  $\phi:\partial D^n\rightarrow S^{n-1}, \phi(x)=\frac{V(z)}{|V(z)|}.$
  The following theorem is a well known result on the index of vector fields, see for example \cite{BG}.

  \begin{Thm}\label{PH}Consider a smooth vector field $\frac{dx}{dt}=V(x).$ If D is a disk containing finitely many zeros $x_1,\cdots ,x_k$ of $V$, then the degree of $\frac{V (x)}{\|V (x)\|}$ on $\partial D$ is equal to the sum of the indices of $V$ at the $x_i$. Moreover, when $x_i$ are all non-degenerate, then 
   $$\sum_{V(x)=0, x\in D}\text{sign}(\text{det}(J))(x)=Q,$$
 where $J$ is the Jacobian of $V$ at $x$ and  $Q$ is the degree of the map $\frac{V(x)}{\|V(x)\|}$ from the boundary of $D$ to the $n-1$ sphere.   
  \end{Thm}
  
\subsection{Center manifold}
\begin{Thm}[Hirsch, Pugh and Shub, \cite{HPS};P16,\cite{Car} ]\label{linear1}Let $E$ be an open subset of $\mathbb{R}^n$ containing the origin and consider the non-linear system $\dot{x}=V(x)$, i.e.,
\begin{equation}\label{centerlinear000}\begin{cases}
\dot{x}=Cx+F(x,y)\\
\dot{y}=My+G(x,y)
\end{cases}
\end{equation}
where the square matrix $C$ has $c$-eigenvalues with zero real parts and the square matrix $M$ has eigenvalues with only non-zero real parts and 
$$F(0,0)=0, DF(0,0)=0; G(0,0)=0, DG(0,0)=0.$$
Then there exists a $\delta>0$ and a function $h\in C^r(B_{\delta}(0)), h(0)=0, Dh(0)=0$ that defines the local center manifold
$$W^c_{\text{loc}}(0)=\{(x,y)\in \mathbb{R}^c\times\mathbb{R}^s\times\mathbb{R}^u\vert y=h(x),  \text{for} \|x\|\leq \delta\}$$
and satisfies $
Dh(x)[Cx+F(x,h(x)]-Mh(x)-G(x,h(x))=0, |x|\leq \delta
$
and the flow on the center manifold $W^c(0)$ is defined by 
$\dot{u}=Cu+F(x,h(u)).$
\end{Thm}
\begin{Thm}[P155, \cite{K}]\label{reduction}The flow given by the vector field (\ref{centerlinear000}) is locally topologically equivalent near the origin to the product system
\begin{equation}\label{centerlinear1}\begin{cases}
\dot{x}=Cx+F(x,h(x))\\
\dot{y}=My,
\end{cases}
\end{equation}
i.e., there exists a homeomorphism $h$ mapping orbits of the first system onto orbits of the second system, preserving the direction of time. 
\end{Thm}

\section{The index of the fixed points}

As a preparation for the proof of Theorem \ref{BB}, in this section we present a product property for the index of the fixed points.  
Let's consider the vector field
$\dot{x}=V(x),$
with $V(x_0)=0$ and 
 the eigenvalues of $DV(x_0)$ have non-zero real part except for one eigenvalue. Here we  assume $x_0$ is an isolated zero point for $V$. 
 Let $U\subset \mathbb{R}^n$ be a small neighborhood of $x_0$ such that $V(x)\neq 0$. Let $D^n$ be a homeomorphic image of $n$-ball with the natural orientation and $x_0\in D^n\subset \overline{D^n}\subset U$.   According to the definition of the index at $x_0$ of $V$, the index of the zero $x_0$ for $V$ is given by the degree of the map $\xi_{V}(x)=\frac{V(x)}{\|V(x)\|}, x\in \partial D^n$
  where $\partial D^n$ is a ball around $x_0$.  
  
The following lemma builds a relation between the index of the fixed points $x_0$ for the vector field $V$ and the index of  $x_0$ as a zero for the map $V(x)$.   
\begin{Lem}\label{algebraicandindex}The index of the zero point $x_0$ of $V$ equals the index of $x_0$ as a fixed point of the locally defined flow $\phi_t$ for $t>0$ sufficiently small. 
\end{Lem}
\begin{proof}  Let $U\subset \mathbb{R}^n$ be a small neighborhood of $x_0$ such that $V(x)\neq 0$ and $\phi^t(x)\neq x$ for all $x\in U\backslash \{x_0\}$. Let $D^n$ be a homeomorphic image of $n$-ball with the natural orientation and $x_0\in D^n\subset \overline{D^n}\subset U$.   According to the definition of the index at $x_0$ of $V$, it suffices to prove 
$\xi_{V}(x)=\frac{V(x)}{\|V(x)\|}, x\in \partial D^n$
and   $\phi_{\phi^t}(x)=\frac{x-\phi^t(x)}{\|x-\phi^t(x)\|}, x\in \partial D^n$
have the same degree.  Denote by $$\delta:=\min\{\inf\{\|V(x)\| \vert x\in \partial D^n\}, \inf\{\|x-\phi^t(x)\| \vert x\in \partial D^n\}\}.$$
Since the eigenvalues of $DV(x_0)$ have non-zero real part except for one eigenvalue, there is no small periodic orbits in $U$. Hence $\delta>0$. As long as $t$ is sufficiently small, we have
$$\|V(x)-x-\phi^t(x)\|=\|V(x)-x-\phi^t(x)\|\leq  \|V(x)-tV(x)\|\leq (1-t)\delta$$
on $\partial D^n$, since $\phi^t$ is differentiable at $x_0$ and $V(x)$ is its differential. Hence $\xi_{V}$ and $\phi_{\phi^t}$ are never antipodal, hence straight-line homotopic via
$\frac{t\xi_{V}+(1-t)\phi_{\phi^t}}{\|t\xi_{V}+(1-t)\phi_{\phi^t}\|}.$
Thus $\deg\xi_{V}=\deg\phi_{\phi^t}.$ \end{proof}

We note here that the vector fields $V$ and $A^{-1}V(A)$ have the same index at the fixed point $x_0$ and $A^{-1}x_0$ respectively, where $A$ is a linear isomorphism. 
This follows immediately from the independence of the definition of index on the coordinates. Please refer to Chapter 7 in \cite{BG} for a proof. 
Under suitable coordinates, we assume the vector field $V$ can be written as
$\label{centerlinear}\begin{cases}
\dot{x}=Cx+F(x,y)\\
\dot{y}=My+G(x,y)
\end{cases}$
where the square matrix $C$ has $c$-eigenvalues with zero real parts and the square matrix $M$ has eigenvalues with only non-zero real parts and 
$F(0,0)=0, DF(0,0)=0; G(0,0)=0, DG(0,0)=0.$
By Theorem \ref{linear1}, there exists a $\delta>0$ and a function $h\in C^r(B_{\delta}(0))$, $h(0)=0, Dh(0)=0$ such that the vector field on the center manifold is defined by 
$$\dot{u}=V^c:=Cu+F(x,h(u)).$$

\begin{Lem}\label{indexproduct}The product property $\text{ind}_{V}(0)=\text{ind}_{V^c}(0)\times (-1)^{\sharp\{i\vert \lambda_i>0\}}$ holds, where $\lambda_i$ are the non-zero eigenvalues for $DV$.
\end{Lem}
\begin{proof}

On the one hand, by Theorem \ref{reduction}, the index of $(0,0)$ for the flow $\phi^t_V$ given by $V$ is the same as the index of $(0,0)$ for the flow $\phi^t_{V_1}$. On the other hand, by Lemma \ref{algebraicandindex}, we obtain $\phi^t_{V}$ and $\phi^t_{V_1}$ have the same index at $(0,0)$. Therefore, the two vector fields  
$\label{centerlinear}
V=\begin{cases}
\dot{x}=Cx+F(x,y)\\
\dot{y}=My+G(x,y)
\end{cases}$
 and  $V_1=(Cx+F(x,h(x)), My)$ have the same index for the zero $(0,0)$.
Finally, by the fact that the index of a product map is the product of the index along each direction, we finish the proof. \end{proof}

\section{The observations on one-dimensional case}
Before we delve into the proof of Theorem \ref{BB}, let's turn our attention to the one-dimensional case first. Theorem \ref{centerlinear} and Lemma  \ref{indexproduct} show that the one-dimensional center direction can reflect 
the bifurcation properties and the index around the fixed point of an arbitrary-dimensional vector field. Following this idea, a classical argument will be the method of center reduction. By doing center reduction, one can change the high-dimensional problems to be one-dimensional problems.  In this section, we study some observations for the one-dimensional case.

 \begin{Lem}\label{existence}Consider the family of smooth functions $V(u,\eps),u\in\mathbb{R}^1,\eps\in\mathbb{R}^1$. Let $u=0$ be an isolated non-hyperbolic zero with non-zero index for $V(u,0).$  Assume
 $$\frac{\partial V}{\partial\eps}(0,0)=0 \text{ and }\frac{\partial^2 V}{\partial u\partial\eps}(0,0)\neq 0.$$
Then for any $\eps$ sufficiently close to zero, we have $u_{\eps}$ as zeros for $V(\cdot,\eps)$ inside $B_{\eps^{1-\delta}}(0)$
, for any sufficiently small number $\delta>0$. 
Moreover, the index of $u_{\eps}$ for $V(\cdot,\eps)$ has the different sign for $\eps>0$ and $\eps<0$. 
\end{Lem}
\begin{proof}We shall use Newton's method to find the zero point $u_{\eps}$. By the assumption that $u=0$ is a non-hyperbolic zero for $V(u,0)$, we get $V(0,0)=0,\frac{\partial V}{\partial u}(0,0)=0.$ Since the index of $u=0$ is non-zero, we know the first $k$ such that $\frac{\partial^kV}{\partial u^k}(0,0)\neq0$ should be odd. Hence $\frac{\partial^2V}{\partial u^2}(0,0)=0.$
Fix an arbitrary small number $\eps$. Denote by $V_{\eps}(u):=V(u,\eps)$. Consider the following sequence of iterations given in Newton's argument:
$u_n=u_{n-1}-\frac{V_{\eps}(u_{n-1})}{V'_{\eps}(u_{n-1})}.$
Then the fixed point of the following map will be the zero points for $V_{\eps}$:
$F_{\eps}(u)=u-\frac{V_{\eps}(u)}{V_{\eps}'(u)}.$
We claim that $F_{\eps}$ is a contracting map on the disc $B_{\eps^{1-\delta}}(0)$. Actually, we have 
\begin{eqnarray*}
F'_{\eps}(u)=1-\frac{V'_{\eps}(u)^2-V_{\eps}(u)V''_{\eps}(u)}{V_{\eps}'(u)^2}
=\frac{V_{\eps}(u)V''_{\eps}(u)}{(V'_{\eps}(u))^2}.\end{eqnarray*} 
Denote by $\frac{\partial^2 V}{\partial u\partial\eps}(0,0)=c\neq0.$
The denominator $V'_{\eps}(u)$ satisfies 
\begin{eqnarray*}\vert V'_{\eps}(u)\vert &=&\vert \frac{\partial V}{\partial u}(u,\eps)-\frac{\partial V}{\partial u}(0,\eps)+\frac{\partial V}{\partial u}(0,\eps)-\frac{\partial V}{\partial u}(0,0)\vert
\\ &\geq&-\vert \frac{\partial V}{\partial u}(u,\eps)-\frac{\partial V}{\partial u}(0,\eps)\vert+\vert\frac{\partial V}{\partial u}(0,\eps)-\frac{\partial V}{\partial u}(0,0)\vert\\
&\geq&\frac{\partial^2V}{\partial u\partial \eps}(0,\tilde{\eps})\eps-\frac{\partial^2V}{\partial u^2}(\tilde{u},\eps)u
\geq C_0(\vert c+\tilde{\eps}\vert)\eps-\vert\tilde{u}u\vert
\geq Cc\eps
\end{eqnarray*} on the ball  $B_{\eps^{1-\delta}}(0)$, where $C,C_0$ are  constant numbers (in the following argument we shall use $C$ for all constant numbers). 
Similarly,  the numerator satisfies
$V_{\eps}(u)V''_{\eps}(u)
\leq Cc\eps^{3-2\delta}.$
Therefore we have
\begin{eqnarray*}
F'_{\eps}(u)=1-\frac{V'_{\eps}(u)^2-V_{\eps}(u)V''_{\eps}(u)}{V_{\eps}'(u)^2}
=\frac{V_{\eps}(u)V''_{\eps}(u)}{(V'_{\eps}(u))^2}
\leq\frac{Cc\eps^{3-2\delta}}{Cc^2\eps^2}
\leq C\eps^{1-2\delta},
\end{eqnarray*}
where $C$ is a constant number. Hence we finish the proof of the claim. On the other hand, since
$
|F_{\eps}(u)|\leq C\eps^{2-3\delta}\leq \eps^{1-\delta}, $  we have $F_{\eps}(B_{\eps^{1-\delta}}(0))\subset B_{\eps^{1-\delta}}(0)$ for small $\delta>0$. It follows that there is one and only one fixed point inside $B_{\eps^{1-\delta}}(0)$.   At $u_{\eps}$,  we have $\frac{\partial V(u_{\eps},\eps)}{\partial u}$ has the same sign as $c\eps$. Since there is a change of sign for $c\eps$ with the variation of $\eps$ from negative to positive, there is a change of sign for $\frac{\partial V(u_{\eps},\eps)}{\partial u}$ with the variation of $\eps$ from negative to positive.

\end{proof}

The following lemma shows that the vector field 
has one and only one equilibrium at one side of the bifurcation time.
\begin{Lem}[Uniqueness]\label{uniqueness}Consider the family of one dimensional 
vector filed:
$\dot{u}=V(u,\eps), u\in\mathbb{R}^1,\eps\in\mathbb{R}^1.$ Assume $u=0$ to be an isolated non-hyperbolic zero with non-zero index for $V(u,0)$.  Assume 
$$\frac{\partial V}{\partial\eps}(0,0)=0 \text{ and }\frac{\partial^2V}{\partial\eps\partial u}(0,0)\neq 0.$$
Then there exist a neighborhood $U\subset\mathbb{R}^1$ of $x=0$ and a small number $\eps_0>0$ such that there is one and only one zero $u_{\eps}\in U$ for any $\eps\in[0,\eps_0]$ or any $\eps\in [-\eps_0,0]$. 
\end{Lem}
\begin{proof}
By the implicit theorem and the assumption $ \frac{\partial^2V}{\partial\eps\partial u}(0,0)\neq 0,$ there exists 
  $(u,\eps(u))$ to be the graph of $\frac{\partial V}{\partial u}(u,\eps(u))=0.$ 
We claim that there exists a small neighborhood $(-r,r)$ such that  either $\eps(u)>0,$ for any $u\in (-r,r)$ or $\eps(u)<0,$ for any $u\in (-r,r)$. Now let's prove this claim.  Since the index  of $V(u,0)$ at $u=0$ is $1$ (the argument for the index $-1$ case is similar ), there exists a small neighborhood $(-r,r)$ such that
$V(u,0)>0, \forall u\in (0,r), V(u,0)<0, \forall u\in (-r,0).$
Hence by the mean value theorem, for any $0<r_1<r$ there exists 
$u_1\in (0,r_1)$ such that $\frac{\partial V(u,\eps)}{\partial u}\vert_{(u_1,0)}>0$ and $u_2\in (0,r_1)$ such that $\frac{\partial V(u,\eps)}{\partial u}\vert_{(u_2,0)}>0.$
On the other hand,  the graph of $\frac{\partial V}{\partial u}(u,\eps)=0$ will cut the $(u,\eps)$ space into two connected region 
$A_1=\{(u,\eps)\vert \frac{\partial V}{\partial u}(u,\eps)>0\}$
and
$A_2=\{(u,\eps)\vert \frac{\partial V}{\partial u}(u,\eps)<0\}.$
Hence the vertical line  $([-r,r],0)\backslash \{(0,0)\}$
can only lie in $A_1$.   So $(u,\eps(u))$ can not go across the vertical line $([-r,r],0)$ and that  finish the claim. 


By the definition of index, we have any zeros of $V(u,\eps)=0$ lying in $A_1$ has index $1$, any zeros of $V(u,\eps)=0$ lying in $A_1$ has index $-1$ and any zeros of $V(u,\eps)=0$ lying in $(u,\eps(u))$ can only have index $1,-1$ or $0$. By Theorem \ref{PH}, we have for sufficiently small $\vert \eps\vert$, $\sum_{V(u,\eps)=0}\text{index}(u)=1.$
If $\eps(u)>0$, we have for any $\eps<0$ sufficiently close to zero, there are no zero points on $(u,\eps(u))$, hence there is one and unique one zero $u(\eps),$ for $\eps<0$. If $\eps(u)<0$, we have for any $\eps>0$ sufficiently close to zero, there is no zero points on $(u,\eps(u))$, hence there is one and unique one zero $u(\eps),$ for $\eps>0$. \end{proof}

\begin{Cor}[Bifurcation]\label{existence1}Consider the  family of one dimensional 
vector filed:
$\dot{u}=V(u,\eps), u\in\mathbb{R}^1,\eps\in\mathbb{R}^1.$ Assume $u=0$ to be an isolated non-hyperbolic zero with non-zero index for $V(u,0)$.  Assume 
$$\frac{\partial V}{\partial\eps}(0,0)=0 \text{ and }\frac{\partial^2V}{\partial\eps\partial u}(0,0)\neq 0.$$
Then $V$ undergoes a $1\rightarrow k$ or $k\rightarrow 1,$ $k\geq 3$ around a neighborhood of $(u_0,\eps_0)$.
\end{Cor}
\begin{proof}By Lemma \ref{existence}, there always exists  zero $x_{\eps}$ for $V(\cdot,\eps)$. By Lemma \ref{uniqueness}, there exists a neighborhood $U$ of $x=0$ such that either $x_{\eps}$ is the only zero for $V(\cdot,\eps)$, for sufficiently close to zero negative $\eps$ or for sufficiently close to zero positive $\eps$. Assume it holds for negative $\eps$. By Lemma \ref{existence} again, the index of $u_{\eps}$ changes sign when $\eps$ varies from negative to positive. 
Hence  there must be at least two other zeros inside $U$ for $\eps>0$.  
\end{proof}

Finally, let's give a criterion for the $1\rightarrow 3$  or $3\rightarrow 1$ bifurcation. The condition  $\frac{\partial^3V}{\partial^3u}(0,0)\neq 0$ in the following corollary plays the role of the multiplicity assumption in Corollary \ref{RScor}. 

\begin{Cor}[Pitchfork Bifurcation]\label{onethree}Consider the  family of one dimensional 
vector filed:
$\dot{u}=V(u,\eps), u\in\mathbb{R}^1,\eps\in\mathbb{R}^1.$ Assume $u=0$ to be an isolated non-hyperbolic zero with non-zero index for $V(u,0)$.  Assume 
$$\frac{\partial V}{\partial\eps}(0,0)=0, \frac{\partial^2V}{\partial\eps\partial u}(0,0)\neq 0, \text{ and }\frac{\partial^3V}{\partial^3u}(0,0)\neq 0.$$
Then $V(x,\eps)$ undergoes a $1\rightarrow 3$ or $3\rightarrow 1$ bifurcation around $(0,0)$.
\end{Cor}
\begin{proof}  Since $\frac{\partial^3V}{\partial^3u}(0,0)\neq 0$, locally the maximal number of zeros is $3$. By Corollary \ref{existence1},  it undergoes a $1\rightarrow 3$  or $3\rightarrow 1$bifurcation. We finish the proof. 
\end{proof}

\section{The undergoing of bifurcations}
In this section, we present the proof of the undergoing of bifurcation under the assumptions (P0), (P1) and (P2), i.e., the proof of Theorem \ref{BB}. First of all, let's study the invariance of (P0),(P1) and (P2) under the change of coordinates. In the following argument, we shall use an equivalent condition for (P1):
\begin{enumerate}
\item[(P1')]$v_l\frac{\partial V}{\partial \eps}=0$, where $v_lD_xV(x_0,\eps_0)=0$. 
\end{enumerate}
The following lemma shows that the assumption (P2) makes sense. 
\begin{Lem}For the vector filed $V(x,\eps)$ with the conditions (P0) and (P1), there exists $\omega=(\omega_1,\cdots,\omega_n,\omega_{n+1})^{\top}$ such that 
$DV(x_0,\eps_0)\omega=0 \text{ and } \omega_{n+1}\neq 0. $

\end{Lem}
\begin{proof}Denote by $v_l$ and $v_r$  the vectors such that $$v_lD_x V(x_0, \eps_0)=0 \text{ and } D_x V(x_0, \eps_0)v_r=0.$$
It is straightforward that  $(v_l,0)DV=0.$ Assume the extended vector fields to be $
\dot{x}=V(x,\eps),
\dot{\eps}=0.$
Differentiating the extended vector field, we have $\begin{bmatrix}D_xV&D_{\eps}V\\0&0
\end{bmatrix}$
with $(v_l,0)$ and $(0,1)$ as two left eigenvectors for the eigenvalue zero. Since the dimension of the left null space and the right null space are the same, there exists an vector $\omega=(\omega_1,\cdots,\omega_n,\omega_{n+1})^{\top}$ such that
$DV(x_0,\eps_0)\omega=0 \text{ and } \omega_{n+1}\neq 0.$ 
\end{proof}

\begin{Lem}\label{invariantBB}For a family of vector fields $\dot{x}=V(x,\eps)$, the following conditions:
 \begin{enumerate}
\item[(P0)] $V(x,\eps)$  has an isolated simple non-hyperbolic equilibrium $(x_0,\eps_0)$ with non-zero index. Denote by $v_l$ and $v_r$  the unique left eigenvector for the eigenvalue $0$, i.e., $v_lD_x V(x_0, \eps_0)=0 \text{ and } D_x V(x_0, \eps_0)v_r=0.$
\item[(P1)]  $v_l\frac{\partial V}{\partial\eps}\vert_{(x_0,\eps_0)}=0;$
\item[(P2)]$D(\text{det}(D_xV))(x_0,\eps_0)\omega\neq 0,$ for any $\omega=(\omega_1,\cdots,\omega_{n+1})^{\top}$ such that 
$$DV(x_0,\eps_0)\omega=0\text{ and }\omega_{n+1}\neq 0,$$ 
 \end{enumerate}
are invariant under the linear change of coordinates $\tilde{A}=
\begin{bmatrix}A&\ast \\
0&1\end{bmatrix}.$ 
\end{Lem}
\begin{proof}
Consider the following linear change of coordinates:
$(\tilde{x},\tilde{\eps})^{\top}=\tilde{A}(x,\eps)^{\top},$
where
$\tilde{A}=
\begin{bmatrix}A&\ast \\
0&1\end{bmatrix}.$
Denote by $\tilde{A}^{-1}$ the inverse matrix of $\tilde{A}$. Since the inverse of the upper triangular matrix are still upper triangular, we have
$\tilde{A}^{-1}=
\begin{bmatrix}A^{-1}&\ast \\
0&1\end{bmatrix}.$ 
 Under the new coordinates, the vector field becomes
$
(\dot{\tilde{x}})^{\top}=AV(\tilde{A}^{-1}(\tilde{x},\tilde{\eps})^{\top}):=\tilde{V}(\tilde{x},\tilde{\eps}).$
Denote by $\omega^1:=\begin{bmatrix}v_r\\0\end{bmatrix}, \omega=(\omega_1,\cdots,\omega_n,\omega_n)^{\top}$ the base for the kernel of $D_{(x,\eps)}V(x_0,\eps_0)$.
Since 
$D_{(\tilde{x},\tilde{\eps})}\tilde{V}(\tilde{x}_0,\tilde{\eps}_0)=ADV(\tilde{A}^{-1}(\tilde{x},\tilde{\eps})^{\top})\tilde{A}^{-1},$
we have 
\begin{eqnarray*}ADV(\tilde{A}^{-1}(\tilde{x},\tilde{\eps})^{\top})\tilde{A}^{-1}\tilde{A}\begin{bmatrix}v_r\\0\end{bmatrix}=ADV(\tilde{A}^{-1}(\tilde{x},\tilde{\eps})^{\top})\begin{bmatrix}v_r\\0\end{bmatrix}=0\end{eqnarray*}
and  $ADV(\tilde{A}^{-1}(\tilde{x},\tilde{\eps})^{\top})\tilde{A}^{-1}\tilde{A}\omega=ADV(\tilde{A}^{-1}(\tilde{x},\tilde{\eps})^{\top})\omega=0. $
Hence the base for the center direction of the kernel $D_{(\tilde{x},\tilde{\eps})}\tilde{V}(\tilde{x}_0,\tilde{\eps}_0)$ is
$\{\tilde{A}\omega^1=\begin{bmatrix}Av_r\\0\end{bmatrix},\tilde{A}\omega\}.$ 
For  the vector field $\tilde{V}$, let's check the conditions (P0),(P1) and (P2). Assume $(\tilde{x}_0,\eps_0)$ to be the fixed points. 
Actually, the first condition (P0) $\text{index}(\tilde{x}_0)=\text{index}(x_0)$ holds, since index is topological invariant.  

Let's verify (P1). First of all, the left eigenvector of $\tilde{V}$ for the eigenvalue zero is given by 
$v_lA^{-1}.$
Hence we have $$v_lA^{-1}D_{\tilde{\eps}}\tilde{V}\vert_{\tilde{x}_0}=v_lA^{-1}AD_{\tilde{\eps}}(V(\tilde{A}^{-1}(\tilde{x},\eps))\tilde{A}^{-1}(0,1)^{\top}
=0. $$
Now let's check the condition (P2) for the vector field $\tilde{V}$. For this vector field, we have 
$
D_{\tilde{x}}\tilde{V}=AD_xV(\tilde{A}^{-1}(\tilde{x},\eps))A^{-1}.$
Moreover, it follows that
\begin{eqnarray*}
\det(D_{\tilde{x}}\tilde{V})=\det(AD_xV(\tilde{A}^{-1}(\tilde{x},\eps))A^{-1})
=\det D_xV(\tilde{A}^{-1}(\tilde{x},\tilde{\eps})).
\end{eqnarray*}
Hence we have
$
D(\text{det}(D_{\tilde{x}}(\tilde{V}))(\tilde{x},\tilde{\eps})=D\det(D_xV)(\tilde{A}^{-1}(\tilde{x},\tilde{\eps}))\tilde{A}^{-1}.
$
For any $\tilde{\omega}=(\tilde{\omega}_0,\cdots,\tilde{\omega}_{n+1})^{\top}$, we have the $$\omega=\tilde{A}^{-1}\tilde{\omega}=\begin{bmatrix}A^{-1}(\omega_1,\cdots,\omega_n)^{\top}+\omega_{n+1}\ast, \\\omega_{n+1}\end{bmatrix}.$$ 
So  $\omega_{n+1}\neq 0$ if and only if $\tilde{\omega}_{n+1}\neq 0.$ 
On the other hand, we have
\begin{eqnarray*}
D(\text{det}(D_{\tilde{x}}\tilde{V})(\tilde{x}_0,\tilde{\eps}_0))\tilde{\omega}&=&D\det(D_xV)(\tilde{A}^{-1}(\tilde{x}_0,\tilde{\eps}_0))\tilde{A}^{-1}(\tilde{A}\omega-t\tilde{A}v_r)\\
&=&D\det(D_xV)(\tilde{A}^{-1}(\tilde{x}_0,\tilde{\eps}_0))\omega\\
&=&D\det(D_xV)(x_0,\eps_0)))\omega\neq 0.\end{eqnarray*}
Hence (P2) still holds. 
\end{proof}

Now we are ready to present the proof of Theorem \ref{BB}.

\begin{proof}[The proof of Theorem \ref{BB}]

By Theorem \ref{PH}, 
a continuous deformation would not change the total index in $U$,
that is, $\mathrm{index}(V(\cdot,\eps), U)=\text{index}(x_0).$
By Lemma \ref{invariantBB},  we can assume the vector field $V$ is of the form $\begin{cases}\dot{u}=F(u,y,\eps)\\
\dot{y}=My+G(u,y,\eps)
\end{cases}$
where $u\in\mathbb{R}^1$ and $y\in\mathbb{R}^{n-1}$, the square matrix $M$ has eigenvalues with only non-zero real parts and 
$$F(0,0,0)=0, D_{(u,y)}F(0,0,0)=0, G(0,0,0)=0, DG(0,0,0)=0,$$
with the conditions (P0),(P1) and (P2). 
The left center direction for $V$ now is $v_l=(1,0)$. By (P1), we have 
$v_lD_{\eps}V(0,0,0)=D_{\eps}F(0,0,0)=0.$ Hence we have $DF(0,0,0)=0$. So we can apply Theorem \ref{linear1} to the extended vector field by adding $\dot{\eps}=0$ as one direction. 
By Theorem \ref{linear1}, there exists a smooth function $h(u,\eps)$ which represents the center manifold for $V$. 
The vector field along the center becomes 
$\dot{u}=F(u,h(u,\eps),\eps).$
By Lemma \ref{indexproduct}, the index of $(u,y)=(0,0)$ for $V(u,y,0)$ is non-zero if and only if  the index of $u=0$ is non-zero for $F(u,h(u,0),0)$. 
Hence by (P0) assumption, it follows that the first $k$ such that $\frac{\partial^k F}{\partial u^k}(0,0)\neq 0$ is an odd number and $k\geq 3$.

 \textbf{Claim 1:} $\frac{\partial^2F}{\partial u\partial \eps}(0,0)\neq 0$. 
First of all, let's give some discussion on $DV$. We have 
$$D_{(u,y)}V(u,y,\eps)=\begin{bmatrix}D_uF(u,y,\eps) & D_yF(u,y,\eps)\\
D_uG(u,y,\eps)&M+D_yG(u,y,\eps)
\end{bmatrix}.$$
By Jacobi's formula, 
\begin{eqnarray*}&&\frac{\partial \det(D_{(u,y)}V(u,0,0))}{\partial u}
\\&=&tr(adj\begin{bmatrix}D_uF(u,0,0) & D_yF(u,0,0)\\
D_uG(u,0,0)&M+D_yG(u,0,0)
\end{bmatrix}\begin{bmatrix}D_{uu}F(u,0,0) & D_{yu}F(u,0,0)\\
D_{uu}G(u,0,0)&D_{yu}G(u,0,0)
\end{bmatrix})\\
&=&tr(adj\begin{bmatrix}D_uF(u,0,0) & D_yF(u,0,0)\\
D_uG(u,0,0)&M+D_yG(u,0,0)
\end{bmatrix}\begin{bmatrix}D_{uu}F(u,0,0) & D_{yu}F(u,0,0)\\
D_{uu}G(u,0,0)&D_{yu}G(u,0,0)\end{bmatrix}).
\end{eqnarray*}
 Hence at $u=0$, we have
 \begin{eqnarray*}\frac{\partial \det(D_{(u,y)}V(0,0,0))}{\partial u}&=&tr(adj\begin{bmatrix}0 & 0\\
0&M
\end{bmatrix}\begin{bmatrix}0 & D_{yu}F(0,0,0)\\
D_{uu}G(0,0,0)&D_{yu}G(0,0,0)\end{bmatrix})\\
&=&tr(\begin{bmatrix}\det M & 0\\
0&0
\end{bmatrix}\begin{bmatrix}0 & D_{yu}F(0,0,0)\\
D_{uu}G(0,0,0)&D_{yu}G(0,0,0)\end{bmatrix})\\
&=&0.
\end{eqnarray*}

It is easy to see that the kernel of  $DV(u,y,\eps)$ has $(1,0,0)$ and $(0,0,1)$ as a base. 
Since along the direction $(1,0,0)$, we have 
\begin{eqnarray*}D_{(u,y,\eps)}\det(D_{(u,y)}V)\vert_{(0,0,0)}(1,0,0)^{\top}&=&\frac{\partial \det(D_{(u,y)}V)}{\partial u}\vert_{(0,0,0)}\\
&=&\frac{\partial \det(D_{(u,y)}V(u,0,0))}{\partial u}\vert_{u=0}=0.
\end{eqnarray*}
Hence by assumption (P2), 
 we have
 $$D_{(u,y,\eps)}\det(D_{(u,y)}V)\vert_{(0,0,0)}(0,0,1)^{\top}=\frac{\partial \text{Det} D_xV}{\partial \eps}\vert_{(x_0,\eps_0)}\neq 0. $$
On the other hand, 
$\text{Det}D_{(u,y)}V(0,0,\eps)=\frac{\partial F}{\partial u}(0,0,\eps)(\text{Det}(M+\frac{\partial G}{\partial y}(0,0,\eps))).$
Hence 
\begin{equation*}
\frac{\partial \text{Det} D_xV\vert_{(x_0,\eps_0)}}{\partial \eps}=\frac{\partial^2 F}{\partial u\partial\eps}(0,0,0)(\text{Det}(M+\frac{\partial G}{\partial y}(0,0,0)))
=\frac{\partial^2F}{\partial u\partial \eps}\cdot\text{Det}M\neq0. 
\end{equation*}
So we get $\frac{\partial^2F}{\partial u\partial \eps}\neq 0. $ We complete the proof of Claim 1. Hence the function 
$V^c(u,\eps):=F(u,h(u,\eps),\eps)$
satisfies all of the following conditions in Corollary \ref{existence1}. By Corollary \ref{existence1}, $V$ undergoes a $1\rightarrow k$ or $k\rightarrow 1$ bifurcation.
\end{proof}

\begin{Rem}
Fix $V\in\mathcal{F}$. For $\eps>0,$ the number of zeros is less than or equal to the first non-vanishing jet of $V(x,0)$ restricted to the center manifold. 
\end{Rem}
At the end of this section, we would like to prove Corollary \ref{RScor} from Theorem \ref{BB}.
\begin{proof}[The proof of Corollary \ref{RScor}] Since $\frac{d\lambda_1}{d\eps}(x_0,\eps_0)>0$, $\lambda_1(x_0,\eps_0)=0$,  we have $\lambda_1(x_0,\eps)<0$, for  $\eps<\eps_0$ and close enough to $\eps_0$. Hence the index of $\lambda_1(x_0,\eps)\neq0.$ By the isolated requirement on the fixed points $(x_0,\eps)$ for $\eps<\eps_0$ and the stability of the index of fixed points, we know the index of $(x_0,\eps_0)$ is non zero. By $\lambda_1(x_0,\eps_0)=0$ again, we know $\frac{\partial V}{\partial \eps}(x_0,\eps_0)=0$. Then the condition (P1) follows. By $\frac{d\lambda_1}{d\eps}(x_0,\eps_0)>0$ and  $\lambda_1(x_0,\eps_0)=0$, we know $D_x(\det D_xV)=0$ and $D_{\eps}(\det D_xV)\neq 0$. So (P2) holds. 
Hence we have all of the conditions required in Theorem \ref{BB}. It follows that there exists $1$ to $k, k\geq \infty$  bifurcations. By the assumption on the multiplicity, there are at most three fixed points showing up. So it is pitchfork bifurcation. We finish the proof of this corollary. 
\end{proof}

\section{Pitchfork bifurcation and its genericity}

In this section, we shall prove the criterion for pitchfork bifurcation and its genericity. 
To do this, we would like to state an equivalent condition first.
\begin{Lem}\label{P3invariant}For any vector field $V$, assume $(c_1(u),\cdots,c_n(u))$ to be the center manifold. The following condition
\begin{enumerate}
\item[(P3)'] $(\det(D_xV(c_1(u),\cdots,c_n(u))))''\vert_{u_0}\neq 0$
where $u_0$ is the bifurcation point
\end{enumerate}
is equivalent to 
\begin{enumerate}
\item[(P3)''] $(N^{\top}D_x^2\det(D_xV)N+D_x\det(D_xV)N')\vert_{(x_0)}\neq 0$
where $N'=(c_1''(u),\cdots,c_n''(u)).$
\end{enumerate}
\end{Lem}
\begin{proof}This is basically due to the chain rule. Denote by $N=(1,0,\cdots,0)^{\top}$.
It follows from 
\begin{eqnarray*}
&&(\det(D_xV(c_1(u),\cdots,c_n(u))))''
=(D\det(D_xV(c_1(u),\cdots,c_n(u)))(c_1',\cdots,c_n')^{\top})'\\
&=& (c_1',\cdots,c_n')D^2\det(D_xV(c_1(u),\cdots,c_n(u)))(c_1',\cdots,c_n')^{\top}\\
&&+D\det(D_xV(c_1(u),\cdots,c_n(u)))(c_1'',\cdots,c_n'')^{\top}\\
&=&(N^{\top}D_x^2\det(D_xV)N+D_x\det(D_xV)N')\vert_{(x_0,\eps_0)} \neq 0.\end{eqnarray*}
\end{proof}

\begin{Lem}For the vector field $V$ of the form
$V:=\begin{cases}\dot{u}=F(u,y,\eps)\\
\dot{y}=My+G(u,y,\eps)
\end{cases}$
where $u\in\mathbb{R}^1$ and $y\in\mathbb{R}^{n-1}$, the square matrix $M$ has eigenvalues with only non-zero real parts and 
$F(0,0,0)=0, D_{(u,y)}F(0,0,0)=0, G(0,0,0)=0, DG(0,0,0)=0,$
we have (P3) is equivalent to 
\begin{enumerate}\item[(P3)'']
$(N^{\top}D_x^2\det(D_xV)N+D_x\det(D_xV)N')\vert_{(x_0)}\neq 0$
\end{enumerate}
with the center manifold $(u, c_2(u),\cdots, c_n(u))$, $N=(1,0,0)$ and $N'=(0,c_2''(u),\cdots, c_n''(u))$.
\end{Lem}
\begin{proof}First of all, it is easy to see that $N=(1,0,0)$ is the center direction at $(0,0)$. Hence we can  the center manifold for $V$ to be $c(u)=(u,c_2(u),\cdots,c_n(u)).$
Denote by $N'(u)=(c_1''(u),c''_2(u),\cdots, c''_n(u)).$
Then $N'=N'(u_0)$ where $u_0$ is the point such that $c(u_0)=0$.
By the local center manifold theorem, we have
\begin{equation}\label{centerformula00}V(c(u))=a(u)c'(u),
\end{equation} where $a(u): \mathbb{R}^1\rightarrow \mathbb{R}^1$ is the scaling.  At $u=u_0,$ we obtain $a(u_0)=0.$  Differentiating the equation \ref{centerformula00}, we have 
\begin{equation}\label{centerformula1}DV(c(u))c'(u)=a'(u)c'(u)+a(u)c''(u).
\end{equation}
Hence at $u=u_0$, we have $a'(u_0)=0.$ Moreover, differentiating the equation \ref{centerformula1}, we obtain
\begin{equation*}c'(u)^{\top}D^2V(c(u))c'(u)+DV(c(u))c''(u)=a''(u)c'(u)+a'(u)c'(u)+a'(u)c''(u)+a(u)c'''(u).\end{equation*}
Hence at $u=u_0$, we obtain 
$N^{\top}D^2V(0)N+DV(0)N'=a''(u_0)N.$
Hence 
\begin{equation}\label{N'}DV(0)N'=a''(u_0)N-N^{\top}D^2V(0)N.
\end{equation}  Plugging $V$ into equation \ref{N'}, we obtain $N'=(0, M^{-1}G_{uu}(0))$ and $$N^{\top}D^2\det(D_xV)N=D_{uu}(\det(DV)).$$
Hence we know that 
(P3) is equivalent to (P3)'' .
\end{proof}

Now let's give the proof of Theorem \ref{EP}.
\begin{proof}[The proof of Theorem \ref{EP}] Due to the conditions (P0),(P1) and (P2), we have 
\begin{equation}\label{(P1)2}D_{\eps}V(x_0,\eps_0)=0 \text{ and } D_{\eps}(\text{det}(D_xV)(x_0,\eps_0))\neq0.
\end{equation}
Now let's consider the solution of the  implicit function:
$\det(D_xV)(x,\eps(x))=0.$
By Equation \ref{(P1)2} and the implicit function theorem, we have
$$D_x(\eps(x_0))=-\frac{D_{x}(\text{det}(D_xV)(x_0,\eps_0))}{D_{\eps}(\text{det}(D_xV)(x_0,\eps_0))}.$$
Denote the graph of the center manifold by
$c(u)=(u,c_2(u),\cdots,c_n(u)).$ Then the restriction of $(x,\eps(x))$ to the center manifold becomes:
$(c(u), \eps(c(u))).$ We claim two facts: 

\textbf{Claim 1}: $D_u(\eps(c(u)))\vert_{c(u)=x_0}=0;$
 and \textbf{Claim 2:} $D^2_u(\eps(c(u)))\vert_{c(u)=x_0}\neq 0$.

Claim 1 follows from the following equality:
\begin{equation*}D_u(\eps(c(u)))\vert_{c(u)=x_0}=D_x(\eps(c(u)))D_u(c(u))=\frac{D_{x}(\text{det}(D_xV)(x_0,\eps_0))}{D_{\eps}(\text{det}(D_xV)(x_0,\eps_0))}D_u(c(u))=0,
\end{equation*}
where the second equality holds because of the index assumption (as Claim 1 in the proof of Theorem \ref{BB}). 
Moreover, Claim 2 holds because of Lemma \ref{P3invariant}. 
By  Claim 1 and Claim 2, we know locally the graph of $\eps(c(u))$ satisfies either $\eps(c(u))>0, $ or $\eps(c(u))<0$. Without lose of generality, we assume $\eps(c(u))>0$. Hence for sufficiently small $\eps>0$,there exist at most two points on the center manifold such that $\det(D_xV(c(u),\eps))=0$. Moreover, for sufficiently small $\eps>0$,  there are at most three zeros for $V(c(u),\eps)$. Otherwise by the mean value theorem, there will be more than three points with $\det(D_xV(c(u),\eps))=0$ which is a contradiction. 
By Theorem \ref{BB}, there exists at least three points. Hence we have exactly one to three bifurcation, i.e., pitchfork bifurcation. 
 Hence the proof is complete. 
\end{proof}

\begin{Lem}\label{p3}(P3)' in Lemma \ref{P3invariant} is invariant under the linear change $x=A\tilde{x}$.
\end{Lem}
\begin{proof}Assume the change of coordinates to be $x=A\tilde{x}$. Then the vector field $\dot{x}=V(x)$ becomes  
$\dot{\tilde{x}}=A^{-1}V(A\tilde{x}):=\tilde{V}(\tilde{x},\eps).$ 
For this vector field, we have $
D_{\tilde{x}}\tilde{V}=A^{-1}D_xV(A\tilde{x})A.$
Moreover, it follows that
$\det(D_{\tilde{x}}\tilde{V})=\det(A^{-1}D_xV(A\tilde{x},\eps)A)
=\det D_xV(A\tilde{x}).$
Assume $(c_1(u),\cdots,c_n(u))$ to be the center manifold for $V$. Then it follows directly from the invariance of center manifold, the center manifold after changing of coordinates becomes $A^{-1}(c_1(u),\cdots,c_n(u))$. Hence
\begin{eqnarray*}
\det(D_x\tilde{V}(A^{-1}(c_1(u),\cdots,c_n(u))))=\det (D_xV(c_1(u),\cdots,c_n(u))).
\end{eqnarray*}
So we have $(P3)'$ is invariant under changing of coordinates.
\end{proof}

\begin{proof}[The proof of Theorem \ref{PG}] Based on Theorem \ref{EP} and Lemma \ref{p3}, we only need to prove that the vector fields with the condition (P3) are open and dense inside $\mathcal{F}$. 
Since $(0, M^{-1}G_{uu}(0))$ is decided by  $(D_xV, D_x^2V)$ (order two terms in the expansion) of the vector field $V(x)$ at $(x_0,\eps_0)$, where we denote $x=(u,y)$.  
Besides, $D_x(\det(D_xV))\vert_{(x_0,\eps_0)}$ is also determined by $(D_xV,D^2_xV)$ at $(x_0,\eps_0).$ 
On the other hand, since
$D_x^2\det D_xV(x)$ is decided by $(DV,D^2V,D^3V)$, we can just perturb $V$ such that we only change $D^3V$ such that  (P3)
holds. 
Hence we know the maps with (P3) is open and dense inside $\mathcal{F}$.
\end{proof}

Enrique Pujals, Department of Mathematics, Graduate Center, City University of New York, 365 5th Ave, New York, NY 10016. epujals@gc.cuny.edu

Michael Shub, Department of Mathematics, City College of  New York, 160 Convent Ave, New
York, NY 10031. mshub@ccny.cuny.edu

Yun Yang, Department of Mathematics, City College of  New York, 160 Convent Ave, New
York, NY 10031. yyang@gc.cuny.edu
  
 \end{document}